\documentclass[reqno,11pt]{amsart}

\usepackage[letterpaper, margin=1in]{geometry}
\usepackage{microtype}

\usepackage{amsmath, amssymb, amsthm, bm}
\usepackage[shortlabels]{enumitem}

\usepackage{hyperref}

\usepackage[noabbrev,capitalize]{cleveref}
\crefname{equation}{}{}

\newtheorem{theorem}{Theorem}[section]
\newtheorem{lemma}[theorem]{Lemma}
\newtheorem{corollary}[theorem]{Corollary}
\newtheorem{proposition}[theorem]{Proposition}
\newtheorem{conjecture}[theorem]{Conjecture}

\theoremstyle{definition}
\newtheorem{definition}[theorem]{Definition}
\newtheorem{question}[theorem]{Question}

\theoremstyle{remark}
\newtheorem*{remark}{Remark}

\newcommand{\abs}[1]{\left\lvert#1\right\rvert}
\newcommand{\norm}[1]{\left\lVert#1\right\rVert}
\newcommand{\ang}[1]{\left\langle #1 \right\rangle}
\newcommand{\floor}[1]{\left\lfloor #1 \right\rfloor}
\newcommand{\ceil}[1]{\left\lceil #1 \right\rceil}

\newcommand{\RR}{\mathbb{R}}

\DeclareMathOperator{\rank}{rank}
\DeclareMathOperator*{\PSL}{PSL}

\title{Equiangular lines with a fixed angle}

\author[Jiang]{Zilin Jiang}
\author[Tidor]{Jonathan Tidor}
\author[Yao]{Yuan Yao}
\author[Zhang]{Shengtong Zhang}
\author[Zhao]{Yufei Zhao}
\thanks{
Jiang was supported by an AMS Simons Travel Grant and NSF Award DMS-1953946.
Tidor was supported by the NSF Graduate Research Fellowship Program DGE-1745302.
Zhao was supported by NSF Award DMS-1764176, the MIT Solomon Buchsbaum Fund, and a Sloan Research Fellowship.
}

\address{Jiang: Arizona State University, Tempe, AZ, USA}
\email{zilinj@asu.edu}

\address{Tidor, Yao, Zhang, Zhao: Massachusetts Institute of Technology, Cambridge, MA, USA}
\email{\{jtidor,yyao1,stzh1555,yufeiz\}@mit.edu}

\begin{document}

\begin{abstract}
  Solving a longstanding problem on equiangular lines, we determine, for each given fixed angle and in all sufficiently large dimensions, the maximum number of lines pairwise separated by the given angle.
  
  Fix $0 < \alpha < 1$. Let $N_\alpha(d)$ denote the maximum number of lines through the origin in $\mathbb{R}^d$ with pairwise common angle $\arccos \alpha$. Let $k$ denote the minimum number (if it exists) of vertices in a graph whose adjacency matrix has spectral radius exactly $(1-\alpha)/(2\alpha)$. If $k < \infty$, then $N_\alpha(d) = \lfloor k(d-1)/(k-1) \rfloor$ for all sufficiently large $d$, and otherwise $N_\alpha(d) = d + o(d)$. In particular, $N_{1/(2k-1)}(d) = \lfloor k(d-1)/(k-1) \rfloor$ for every integer $k\ge 2$ and all sufficiently large $d$.
  
  A key ingredient is a new result in spectral graph theory: the adjacency matrix of a connected bounded degree graph has sublinear second eigenvalue multiplicity.
\end{abstract}

\maketitle

\section{Introduction} \label{sec:intro}

A set of lines passing through the origin in $\RR^d$ is called \emph{equiangular} if they are pairwise separated by the same angle. Equiangular lines and their variants appear naturally in pure and applied mathematics. It is an old and natural problem to determine the maximum number of equiangular lines in a given dimension. The study of equiangular lines was initiated by Haantjes~\cite{Haa48} in connection with elliptic geometry and has subsequently grown into an extensively studied subject. Equiangular lines show up in coding theory as tight frames~\cite{SH03}. Complex equiangular lines, also known under the name SIC-POVM, play an important role in quantum information theory~\cite{RBSC04}.

The problem of determining $N(d)$, the maximum number of equiangular lines in $\RR^d$, was formally posed by van Lint and Seidel~\cite{LS66}. The exact value of $N(d)$ has been determined for only finitely many $d$ (see \cite{BY14, GKMS16}). A general upper bound $N(d) \le \binom{d+1}{2}$ was shown by Gerzon (see \cite{LS73}). It had remained open for some time whether there is a matching quadratic lower bound, until de Caen~\cite{dC00} gave a remarkable construction showing $N(d) \ge \tfrac29 (d+1)^2$ for $d$ of the form $d = 6\cdot 4^i - 1$, which in particular implies that $N(d) = \Theta(d^2)$ for all $d$. All examples of sets of $\Theta(d^2)$ equiangular lines in $\RR^d$ have angles approaching $90^\circ$ as $d \to \infty$. It turns out that a completely different behavior emerges when the common angle is held fixed as $d \to \infty$, which is the focus of this paper.

Let $N_\alpha(d)$ denote the maximum number of lines in $\RR^d$ through the origin with pairwise angle $\arccos \alpha$. Equivalently, $N_\alpha(d)$ is the maximum number of unit vectors in $\RR^d$ with pairwise inner products $\pm \alpha$. Lemmens and Seidel~\cite{LS73} in 1973 initiated the problem of studying $N_\alpha(d)$ for fixed $\alpha$ and large $d$. They completely determined the values of $N_{1/3}(d)$ for all $d$, and in particular proved that $N_{1/3}(d) = 2(d-1)$ for all $d \ge 15$. Neumann (see \cite{LS73}) showed that $N_\alpha(d) \le 2d$ unless $1/\alpha$ is an odd integer. It was conjectured by Lemmens and Seidel~\cite{LS73} and subsequently proved by Neumaier~\cite{Neu89} that $N_{1/5}(d) = \floor{3(d-1)/2}$ for all sufficiently large $d$. Neumaier~\cite{Neu89} writes that ``the next interesting case [$\alpha = 1/7$] will require substantially stronger techniques.''

We focus on the problem for fixed $\alpha$ and large $d$ and refer the readers to \cite{GY18} for discussion on bounds for smaller values of $d$.

Recently there were a number of significant advances giving new upper bounds on $N_\alpha(d)$, starting with the work of Bukh~\cite{Bukh16} who proved that $N_\alpha(d)$ is at most linear in the dimension for every fixed $\alpha$.\footnote{In fact a stronger version of the inequality was shown by Bukh~\cite{Bukh16}, namely that for every fixed $\beta > 0$ one cannot have more than $C_\beta d$ unit vectors in $\RR^d$ whose mutual inner products lie in $[-1,-\beta]\cup \{\alpha\}$.} Then came a surprising breakthrough of Balla, Dr\"axler, Keevash, and Sudakov~\cite{BDKS18}, who showed that $\limsup_{d \to \infty} N_\alpha(d)/d$, as a function of $\alpha \in (0,1)$, is maximized at $\alpha = 1/3$, and in fact this limit is at most 1.93 unless $\alpha = 1/3$, in which case the limit is 2. In addition to introducing many new tools and ideas, their important paper presents an approach to the equiangular lines problem that forms a bedrock for subsequent work.

An outstanding problem is to determine $\lim_{d \to \infty} N_\alpha(d)/d$ for every $\alpha$. The results in \cite{LS73, Neu89} suggest, and it is explicitly conjectured in \cite[Conjecture 8]{Bukh16}, that $N_{1/(2k-1)}(d) = kd/(k-1) + O_k(1)$ as $d \to \infty$. A conjectural value of $\lim_{d \to \infty} N_\alpha(d)/d$ for every $\alpha$ was given in \cite{JP20} in terms of the following spectral graph quantity.

\begin{definition}[Spectral radius order] \label{def:spectral-order}
  Define the \emph{spectral radius order}, denoted $k(\lambda)$, of a real $\lambda > 0$ to be the smallest integer $k$ so that there exists a $k$-vertex graph $G$ whose spectral radius $\lambda_1(G)$ is exactly $\lambda$. (When we say the spectral radius or eigenvalues of a graph we always refer to its adjacency matrix.) Set $k(\lambda) = \infty$ if no such graph exists.
\end{definition}

Jiang and Polyanskii~\cite{JP20} conjectured that $\lim_{d \to \infty} N_\alpha(d)/d = k/(k-1)$ where $k = k(\lambda)$ with $\lambda = (1-\alpha)/(2\alpha)$. They proved their conjecture whenever $\lambda < \sqrt{2+\sqrt{5}} \approx 2.058$ (the cases $\alpha = 1/3,1/5$, corresponding to $\lambda =1,2$, were known earlier, as discussed). In particular, it was shown that $N_{1/(1+2\sqrt{2})}(d) = 3d/2 + O(1)$. Furthermore, it was shown that $N_\alpha(d) \le 1.49d$ for every $\alpha \notin \{1/3,1/5,1/(1+2\sqrt{2})\}$ and sufficiently large $d > d_0(\alpha)$, improving the earlier bound in \cite{BDKS18}.

There is a natural limitation to all previous techniques when $\lambda \ge \sqrt{2+\sqrt{5}}$, which Neumaier had already predicted at the end of his paper \cite{Neu89} (hence his comment about $\alpha=1/7$, i.e., $\lambda = 3$, mentioned earlier). We refer to \cite{JP20} for discussion.

We completely settle all these conjectures in a strong form.

\begin{theorem}[Main theorem] \label{thm:main}
  Fix $\alpha \in (0, 1)$. Let $\lambda= (1-\alpha)/(2\alpha)$ and $k = k(\lambda)$ be its spectral radius order.
  The maximum number $N_\alpha(d)$ of equiangular lines in $\RR^d$ with common angle $\arccos \alpha$ satisfies
  \begin{enumerate}[(a)]
    \item $N_\alpha(d) = \floor{k(d-1)/(k-1)}$ for all sufficiently large $d > d_0(\alpha)$ if $k < \infty$. \label{thm:main-a}
    \item $N_\alpha(d) = d + o(d)$ as $d \to \infty$ if $k = \infty$. \label{thm:main-b}
  \end{enumerate}
\end{theorem}

\begin{remark}
  Our proof of \ref{thm:main-a} works for $d > 2^{2^{C \lambda k}}$ with some constant $C$. For \ref{thm:main-b}, it is known \cite[Propositions 15 and 23]{JP20} that $d \le N_\alpha(d) \le d+2$ unless $\lambda$ is a totally real algebraic integer that is largest among its conjugates.\footnote{The \emph{conjugates} of an algebraic integer $\lambda$ are the other roots of its minimal polynomial. We say that $\lambda$ is \emph{totally real} if all its conjugates are real.} For the remaining values of $\alpha$, we leave it as an open problem to determine the growth rate of $N_\alpha (d) - d$.
\end{remark}

If $k\ge 2$ is an integer and $\alpha = 1/(2k-1)$, then $\lambda = k-1$ and $k(\lambda) = k$ (the complete graph $K_k$ is the graph on fewest vertices with spectral radius $k-1$), so the following corollary confirms Bukh's conjecture~\cite{Bukh16} in a stronger form, and extending the only two previously known cases of $k=2$ \cite{LS73} and $k=3$ \cite{Neu89}.

\begin{corollary}
  For every fixed integer $k \ge 2$, one has $N_{1/(2k-1)}(d) = \floor{k(d-1)/(k-1)}$ for all sufficiently large $d > d_0(k)$.
\end{corollary}

\section{Proof ideas} \label{sec:sketch}

In this section we summarize several key ideas used in the proof and discuss their origins.

\smallskip

\emph{Connection to spectral graph theory.} Choose a unit vector in the direction of each line in the equiangular set. By considering the Gram matrix, we recast the problem to one concerning the spectrum of the adjacency matrix of an associated graph. The connection between equiangular lines and spectral graph theory has been well known from early works, making equiangular lines one of the foundational problems of algebraic graph theory (e.g., see \cite[Chapter 11]{GR01}).

\smallskip

\emph{Forbidden induced subgraphs.} Using the fact that the Gram matrix is positive semidefinite, we show that the associated graph cannot have certain induced subgraphs. This idea has appeared in the early works of Lemmens and Seidel~\cite{LS66} and Neumaier~\cite{Neu89}, and it was reintroduced in recent papers \cite{BDKS18, Bukh16, JP20} under the guise of taking an orthogonal projection onto some subspace. In our proof, we do not take projections; instead we simply verify the forbidden induced configurations by testing positive semidefiniteness using appropriately chosen vectors.

\smallskip

\emph{Switching}. Given a set of unit vectors representing an equiangular lines configuration, we may negate some unit vector without changing the configuration of lines. The corresponding operation on the associated graph picks some vertex and swaps the adjacency and non-adjacency relations coming from that vertex. The idea of switching already appears in the early work of van Lint and Seidel~\cite{LS66}. It was further used by Neumaier~\cite{Neu89} together with an application of Ramsey's theorem to determine $N_{1/5}(d)$.

A novel extension of the switching argument was introduced in \cite{BDKS18}, combining the knowledge of forbidden induced subgraphs (mentioned above) with an application of Ramsey's theorem. This can be used to show that one can switch some of the vertices in the associated graph so that it has bounded degree. 

\begin{theorem} \label{thm:switch}
  For every $\alpha\in (0, 1)$, there exists some $\Delta$ (depending only on $\alpha$) so that for every set of equiangular lines in $\RR^d$ with common angle $\arccos\alpha$, one can choose a set $S$ of unit vectors, with one unit vector in the direction of each line in the equiangular set, so that each unit vector in $S$ has inner product $-\alpha$ with at most $\Delta$ other vectors in $S$.
\end{theorem}

The proof of this theorem follows by combining Lemmas 2.7 and 2.8 of \cite{BDKS18}. Since this result is an important ingredient of our proof and does not appear explicitly in \cite{BDKS18}, we give a self-contained and streamlined proof in \cref{sec:switch}.

\smallskip

\emph{Second eigenvalue multiplicity.} Our most significant new contribution is an upper bound on the second eigenvalue multiplicity of the associated graph. Let $\lambda_1(G) \ge \lambda_2(G) \ge \dots \ge \lambda_{|G|}(G)$ be the eigenvalues of the adjacency matrix of $G$, accounting for multiplicities as usual. We call $\lambda_j(G)$ the \emph{$j$-th eigenvalue} of $G$.

\begin{theorem} \label{thm:eigen-mult}
  For every $j$ and every $\Delta$, there is a constant $C = C(\Delta,j)$ so that every connected $n$-vertex graph with maximum degree at most $\Delta$ has $j$-th eigenvalue multiplicity at most $Cn/\log\log n$.
\end{theorem}

We only need $j=2$ in this paper, though the proof for any fixed $j$ is essentially the same. The $j$-th eigenvalue multiplicity bound is used in a follow-up work on spherical two-distance sets~\cite{JTYZZ2}.

We introduce a novel approach to bound eigenvalue multiplicity using the Cauchy interlacing theorem along with comparing local and global spectral data via counting closed walks in the graph after deleting a small fraction of the vertices. See \cref{sec:eigen-mult} for the proof as well as remarks on bounds.

In contrast, the strategy in \cite{BDKS18} and later adapted in \cite{JP20} had the flavor of using projections to exclude a finite set of subgraphs with spectral radii exceeding $\lambda$, though this strategy runs into a serious limitation when $\lambda \ge \sqrt{2+\sqrt{5}}$, as foreseen by Neumaier~\cite{Neu89}, since the family of forbidden subgraphs has infinitely many minimal elements~\cite{JP20}. Our method overcomes this significant barrier.

\section{Proof of the main theorem}

A set of $N$ equiangular lines can be represented by unit vectors $v_1, \dots, v_N \in \RR^d$ with $\ang{v_i, v_j} = \pm \alpha$ for all $i \ne j$. The Gram matrix $(\ang{v_i, v_j})_{i,j}$ is a positive semidefinite matrix with $1$'s on the diagonal and $\pm \alpha$ everywhere else, so it is equal to $(1-\alpha) I + \alpha (J - 2A_G)$, where $I$ is the identity matrix, $J$ the all-1's matrix, and $A_G$ the adjacency matrix of an \emph{associated graph $G$} on vertex set $[N]$ where $ij$ is an edge whenever $\ang{v_i, v_j} = -\alpha$. Dividing by $2\alpha$, we can rewrite this matrix as $\lambda I - A_G + \tfrac12 J$, where $\lambda = (1-\alpha)/(2\alpha)$. Conversely, for every $G$ and $\lambda$ for which the above matrix is positive semidefinite and has rank $d$, there exists a corresponding configuration of $N$ equiangular lines in $\RR^d$, one line for each vertex of $G$, with pairwise inner product $\pm \alpha$. Thus the equiangular lines problem has the following equivalent spectral graph theoretic formulation.

\begin{lemma} \label{lem:line-graph}
  There exists a family of $N$ equiangular lines in $\RR^d$ with common angle $\arccos{\alpha}$ if and only if there exists an $N$-vertex graph $G$ such that the matrix $\lambda I - A_G + \frac{1}{2}J$ is positive semidefinite and has rank at most $d$, where $\lambda= (1-\alpha)/(2\alpha)$ and $J$ is the all-1's matrix. \qed
\end{lemma}

We first establish the lower bounds.

\begin{proposition} \label{prop:lower-bounds}
  Let $\alpha\in (0, 1)$ and $\lambda = (1-\alpha)/(2\alpha)$.
  Let $d$ be a positive integer. One has $N_\alpha(d)\ge d$.
  If $k = k(\lambda) < \infty$, then $N_\alpha(d) \ge \floor{k(d-1)/(k-1)}$.
\end{proposition}

\begin{proof}
  Let $G$ be the empty graph on $d$ vertices, so that $A_G = 0$ and $\lambda I - A_G + \frac{1}{2} J$ is positive semidefinite and has rank $d$.
  So $N_\alpha(d) \ge d$ by \cref{lem:line-graph}.
  
  Now assume $k < \infty$. Let $H$ be a $k$-vertex graph with $\lambda_1(H) = \lambda$. Let $G$ be the disjoint union of $\floor{(d-1)/(k-1)}$ copies of $H$ along with $(d-1) - (k-1)\floor{(d-1)/(k-1)}$ isolated vertices. The number of vertices in $G$ is $(d-1) + \floor{(d-1)/(k-1)} = \floor{k(d-1)/(k - 1)}$.
  
  Since $\lambda$ is the spectral radius of $G$ and the multiplicity of $\lambda$ in $G$ is $\floor{(d-1)/(k-1)}$, the matrix $\lambda I - A_G$ is positive semidefinite and has rank $d-1$. Because $\frac{1}{2} J$ is also positive semidefinite and has rank $1$, their sum $\lambda I - A_G + \tfrac12 J$ is positive semidefinite and has rank at most $d$. By \cref{lem:line-graph}, $N_{\alpha}(d) \ge \floor{k(d-1)/(k-1)}$.
\end{proof}

We now prove the upper bounds in \cref{thm:main} assuming \cref{thm:switch,thm:eigen-mult}.

\begin{proof}[Proof of \cref{thm:main}]
  The lower bounds follow from \cref{prop:lower-bounds}. For the upper bounds, consider $N$ equiangular lines in $\RR^d$. By \cref{thm:switch}, there is some constant $\Delta = \Delta(\alpha)$ such that we can choose one unit vector in the direction of each line so that the associated graph (whose edges correspond to negative inner products) has maximum degree at most $\Delta$. Let $C_1,\dots,C_t$ be the connected components of $G$, numbered such that $\lambda_1(G) = \lambda_1(C_1)$.
  
  If $\lambda$ is not an eigenvalue of $A_G$, then $\lambda I - A_G$ has full rank. As $J$ has rank $1$,
  \[
  d \ge \rank (\lambda I - A_G + \tfrac{1}{2}J) \ge N-1.
  \]
  Thus $N \le d + 1$, and \cref{thm:main} clearly holds. Therefore we may assume that $\lambda$ is an eigenvalue of $A_G$.
  
  First consider the case $\lambda_1(G) = \lambda$. By the definition of spectral radius order $k = k(\lambda) < \infty$. Since both $\lambda I - A_G$ and $J$ are positive semidefinite,
  \[
  \ker(\lambda I - A_G + \tfrac{1}{2}J) = \ker(\lambda I - A_G) \cap \ker(J).
  \]
  By the Perron--Frobenius theorem, there is a top eigenvector of $G$ with nonnegative entries. This vector lies in $\ker(\lambda I - A_G)$ but not in $\ker(J)$, implying that $\dim\ker(\lambda I - A_G + \frac{1}{2}J) \le \dim\ker(\lambda I - A_G)-1$. By the rank--nullity theorem, we obtain
  \[
  \rank(\lambda I - A_G) \le \rank(\lambda I - A_G + \tfrac{1}{2}J) - 1 \le d-1.
  \]
  Without loss of generality, suppose $C_1,\dots,C_j$ are the components of $G$ with spectral radius exactly $\lambda$, and thus $\abs{C_1}, \dots, \abs{C_j} \ge k$ by the definition of spectral radius order. By the Perron--Frobenius theorem, the multiplicity of $\lambda$ in each component is at most $1$. Thus
  \[
  \dim\ker(\lambda I - A_G) = j\qquad\text{and}\qquad\rank(\lambda I - A_G) \ge (k - 1)j.
  \]
  Combining the upper and lower bounds on $\rank(\lambda I-A_G)$, we obtain $j \le (d-1)/(k-1)$. Thus,
  \[
  N = \rank(\lambda I - A_G) + \dim\ker(\lambda I - A_G) \le d - 1 + j\le \frac{k(d - 1)}{k - 1}.
  \]
  Therefore \cref{thm:main} holds in this case.
  
  Now we consider the complementary case $\lambda_1(C_1) > \lambda$. Since $\lambda I - A_G + \frac{1}{2}J$ is positive semidefinite and $J$ is a rank 1 matrix, $\lambda I-A_G$ has at most one negative eigenvalue. Thus $\lambda_2(G) \le \lambda$.
  
  We claim that this implies that the spectral radius of all the remaining components is strictly less than $\lambda$. By the Perron-Frobenius theorem, there are top eigenvectors $\bm u, \bm v$ for $C_1, C_i$ with nonnegative entries (positive in the component under consideration and 0 outside it). Since both $\bm 1^{\intercal}\bm u$ and $\bm 1^{\intercal}\bm v$ are positive, we can choose $c\neq 0$ such that $\bm w=\bm u-c\bm v$ satisfies $\bm 1^{\intercal}\bm w=0$. Now since $\lambda I-A_G+\tfrac12J$ is positive semidefinite, we have
  \[
  0 \le \bm w^{\intercal}(\lambda I-A_G+\tfrac12J)\bm w = \bm w^{\intercal}(\lambda I-A_G)\bm w.
  \]
  Expanding and using the fact that the supports of $\bm u$ and $\bm v$ are disconnected in $G$, we find
  \[
  \lambda \bm{u}^\intercal \bm{u} + c^2 \lambda \bm{v}^\intercal \bm{v}
  \ge \bm{u}^\intercal A_{G} \bm{u} + c^2 \bm{v}^\intercal A_{G} \bm{v}
  = \lambda_1(C_1) \bm{u}^\intercal \bm u + c^2 \lambda_1(C_i) \bm{v}^\intercal \bm v,
  \]
  implying that $\lambda_1(C_i)<\lambda$ for all $i>1$.
  Therefore $\lambda I - A_{C_i}$ is invertible for all $i > 1$, so $\dim\ker (\lambda I - A_G) = \dim\ker (\lambda I - A_{C_1}).$ Since $C_1$ has maximum degree at most $\Delta$, \cref{thm:eigen-mult} gives
  \[
  \dim\ker (\lambda I - A_{C_1}) = O_\Delta(\abs{C_1}/\log\log \abs{C_1}) = O_\Delta(N/\log\log N).
  \]
  Also,
  \[
  \rank(\lambda I - A_G) \le \rank(\lambda I - A_G + \tfrac{1}{2}J) + 1 \le d + 1.
  \]
  Thus
  \[
  N = \rank(\lambda I - A_G) + \dim\ker (\lambda I - A_G) \le O_\Delta(N/\log\log N) + d+1.
  \]
  This implies that $N \le d + O_\Delta(d/\log\log d)$. When $k < \infty$, this is smaller than $\floor{k(d-1)/(k-1)}$ for sufficiently large $d$.
\end{proof}

\section{Bounding eigenvalue multiplicity} \label{sec:eigen-mult}

In this section we prove \cref{thm:eigen-mult}, which bounds the $j$-th eigenvalue multiplicity of a connected bounded degree graph.

\begin{definition}
  The $r$-\emph{neighborhood} of a vertex $v$ in a graph $G$, denoted $G_r(v)$, is the subgraph of $G$ induced by all the vertices that are at most distance $r$ away from $v$. An $r$-\emph{net} in $G$ is a subset $V$ of the vertices such that all vertices in $G$ are within distance $r$ from some vertex in $V$.
\end{definition}

\begin{lemma} \label{lem:small-dominating}
  Let $n$ and $r$ be positive integers. Every $n$-vertex connected graph has an $r$-net with size at most $\ceil{n/(r+1)}$.
\end{lemma}

\begin{proof}
  It suffices to prove the lemma in the case where $G$ is a tree. Pick an arbitrary vertex $w$. Take a vertex $v$ at the maximum distance $D$ from $w$. If $D \le r$, then $\{w\}$ is an $r$-net. Otherwise, let $u$ be the vertex on the path between $w$ and $v$ at distance $r$ from $v$. Add $u$ to the net and repeat the argument on the component of $w$ in $G - u$, which has at most $n - r - 1$ vertices.
\end{proof}

The next lemma tells us that removing an $r$-net from a graph significantly decreases its spectral radius.

\begin{lemma} \label{lem:reduce-eigen}
  Let $r$ be a positive integer. If $H$ (with at least 1 vertex) is obtained from a graph $G$ by deleting an $r$-net of $G$, then \[\lambda_1(H)^{2r} \le \lambda_1(G)^{2r} - 1.\]
\end{lemma}

\begin{proof}
  It suffices to prove the lemma in the case where $G$ has no isolated vertices. The result then follows from the Perron--Frobenius theorem and the observation that $A_H^{2r} \le A_G^{2r} - I$ entry-wise (padding zeros to extend $A_H$ to a $\abs{G} \times \abs{G}$ matrix). Indeed, for each vertex $v$ of $H$, the number of closed walks of length $2r$ starting from $v$ is strictly more in $G$ than in $H$, since in $G$ one can walk to a nearest vertex in the $r$-net and then walk back (and then walking back and forth along a single edge to reach length $2r$) and this walk is not available in $H$.
\end{proof}

The next lemma connects the spectrum of a graph with its local spectral radii.

\begin{lemma} \label{lem:eigen-bound}
  For every graph $G$ and positive integer $r$,
  \[
  \sum_{i=1}^{\abs{G}} \lambda_i(G)^{2r} \le \sum_{v\in V(G)} \lambda_1(G_r(v))^{2r}.
  \]
\end{lemma}

\begin{proof}
  The left-hand side counts the number of closed walks of length $2r$ in $G$. The number of such walks starting at $v \in V(G)$ is $\bm 1_v^\intercal A_{G_r(v)}^{2r} \bm 1_v$ since such a walk must stay within distance $r$ from $v$. This quantity is upper bounded by $\lambda_1(G_r(v))^{2r}$, completing the proof.
\end{proof}

\begin{proof}[Proof of \cref{thm:eigen-mult}]
  Let $G$ be a connected $n$-vertex graph with maximum degree at most $\Delta$. If $\lambda_j(G) \le 0$, the theorem holds as the graph has bounded size. Indeed, in this case,
  \[
  2\abs{E(G)} = \sum_{i=1}^n\lambda_i(G)^2
  \le \sum_{i=1}^{j-1}\lambda_i(G)^2+\left(\sum_{i=j}^n\lambda_{i}(G)\right)^2=\sum_{i = 1}^{j - 1}\lambda_i(G)^2 + \left(\sum_{i = 1}^{j - 1}\lambda_i(G)\right)^2 \le j^2\Delta^2.
  \]
  
  Now suppose $\lambda =\lambda_j(G)  > 0$. Let $r_1 = \floor{c \log\log n}$ and $r_2 = \floor{c\log n}$ where $c = c(\Delta,j) > 0$ is a sufficiently small constant. Let $r = r_1+r_2$.
  
  Define $U = \{v \in V(G) : \lambda_1(G_r(v)) > \lambda\}$. We wish to bound the size of $U$. Let $U_0$ be a maximal subset of $U$ such that the pairwise distance (in $G$) between any two elements of $U_0$ is at least $2(r+1)$. 
  Then the graph $G_r(U_0)$ induced by the $r$-neighborhood of $U_0$ has $\abs{U_0}$ connected components each with spectral radius greater than $\lambda$.
  Hence $\lambda_{\abs{U_0}}(G_r(U_0)) > \lambda = \lambda_j(G)$, 
  	and thus $\abs{U_0} < j$ by the Cauchy interlacing theorem.
  Due to the maximality of $U_0$, its $2(r+1)$-neighborhood contains $U$, 
  	and hence $\abs{U} \le \abs{U_0} \Delta^{2(r+1)} < j\Delta^{2(r+1)}$.
  
  Let $V_0$ be an $r_1$-net of size at most $\ceil{n/(r_1 + 1)}$ in $G$ obtained from \cref{lem:small-dominating}.
  Let $H$ be the graph obtained from $G$ after removing $V_0 \cup U$.
  For each $v \in V(H)$, the vertices in $G_r(v)$ not in $H_{r_2}(v)$ form an $r_1$-net of $G_r(v)$, and hence by \cref{lem:reduce-eigen}, $\lambda_1(H_{r_2}(v))^{2r_1} \le \lambda_1(G_r(v))^{2r_1} - 1 \le \lambda^{2r_1} - 1$. By \cref{lem:eigen-bound},
  \[
  \sum_{i=1}^{\abs{H}} \lambda_i(H)^{2r_2} \le \sum_{v\in V(H)} \lambda_1(H_{r_2}(v))^{2r_2} \le \left(\lambda^{2r_1} - 1\right)^{r_2/r_1} n.
  \]
  Hence the multiplicity of $\lambda$ in $H$ is at most
  \[
  \left(1 - \lambda^{-2r_1}\right)^{r_2/r_1} n
  \le e^{-r_2\lambda^{-2r_1}/r_1} n
  \le e^{-\sqrt{\log n}} n,
  \]
  provided that $c$ is chosen to be small enough initially (here we note that $\lambda \le \lambda_1(G) \le \Delta$).  
  Since $\abs{V_0} + \abs{U} \le \ceil{n/(r_1 + 1)} + j\Delta^{2(r+1)} = O_{j,\Delta}(n/\log\log n)$,
  the Cauchy interlacing theorem implies that the multiplicity of $\lambda$ in $G$ is at most $O_{j,\Delta}(n / \log\log n)$.
\end{proof}

\begin{remark}
  \cref{thm:eigen-mult} fails for disconnected graphs since $\lambda_j(G)$ can be the spectral radius of many identical small components. 
  
  It seems likely that the upper bound can be further improved. It cannot be improved beyond $O(n^{1/3})$ due to the following construction: let $p \ge 5$ be a prime and $G$ the Cayley graph of $\PSL(2,p)$ with two standard group generators. Then $G$ is a connected 4-regular graph on $p(p^2-1)/2$ vertices. Since all non-trivial representations of $\PSL(2,p)$ have dimension at least $(p-1)/2$, all eigenvalues of $G$ except $\lambda_1(G)$ have multiplicity at least $(p-1)/2$ (see~\cite{DSV03}). More generally, one can use quasirandom groups~\cite{Gow08}, which are groups with no small irreducible non-trivial representations.
  
  The claim is false without the maximum degree hypothesis. Paley graphs have $p$ vertices and second eigenvalue $(\sqrt{p}-1)/2$ with multiplicity $(p-1)/2$. Other strongly regular graphs and distance-regular graphs have similar properties.
\end{remark}

\section{Switching to a bounded degree graph} \label{sec:switch}

It remains to prove \cref{thm:switch}, which says that one can choose the unit vectors for the equiangular lines so that the associated graph $G$ has bounded degree. Recall that the edges of $G$ correspond to pairs of unit vectors with inner product $-\alpha$. This argument essentially appears in \cite{BDKS18} though phrased differently. Here we give a self-contained and streamlined proof. 

We begin by using the positive semidefiniteness of the Gram matrix to show that certain induced subgraphs cannot appear in $G$.  

\begin{lemma} \label{lem:bounded-clique}
  Let $\alpha \in (0, 1)$. Let $G$ be the associated graph of a set of unit vectors with pairwise inner products $\pm \alpha$. Then the largest clique in $G$ has size at most $\alpha^{-1} + 1$.
\end{lemma}

\begin{proof}
  Let $v_1, \dots, v_M$ be unit vectors corresponding to a clique in $G$, so that $\ang{v_i, v_j} = -\alpha$ for $i \ne j$. Then $0 \le \norm{v_1 + \cdots + v_M}_2^2 = M - M(M-1)\alpha$. Hence $M \le \alpha^{-1} + 1$.
\end{proof}

\begin{definition} \label{def:C_X}
  For a graph $G$ and sets $A\subseteq X\subseteq V(G)$, define $C_X(A)$ to be the set of vertices in $V(G)\setminus X$ that are adjacent to all vertices in $A$ and not adjacent to any vertices in $X\setminus A$.
\end{definition}

\begin{lemma}\label{prop:independent}
  Let $\alpha \in (0, 1)$ and $\lambda = (1-\alpha)/(2\alpha)$. There exist positive integers $M_1, M_2$ depending only on $\alpha$ such that the following holds. Let $G$ be the associated graph of a set of unit vectors with pairwise inner products $\pm \alpha$. If $X$ is an independent set of $G$ with at least $M_1$ vertices, then
  \begin{enumerate}[(a)]
    \item the maximum degree of the subgraph of $G$ induced by $C_X(\varnothing)$ (i.e., the non-neighbors of $X$) is at most $\lceil\lambda^2\rceil$, and
    \item $\abs{C_X(Y)} \le M_2$ for every nonempty proper subset $Y$ of $X$.
  \end{enumerate}
\end{lemma}

\begin{proof}
  (a) Assume for contradiction that there exists a star $K_{1,D}$ in $C_X(\varnothing)$ with vertex set $V_1$ where $D = \lceil\lambda^2\rceil + 1$. Consider the vector $\bm v$ that assigns $\sqrt{D}$ to the center of the star, $1$ to all other vertices in $V_1$, $-(D + \sqrt{D})/\abs{X}$ to all vertices in $X$, and 0 to all other vertices of $G$. We have
  \[
  \bm v^\intercal\left(\lambda I - A_G + \tfrac{1}{2}J\right)\bm v \ge 0
  \]
  due to positive semidefiniteness. Since $J\bm v = 0$,
  \[
  0\le\lambda(\bm v^\intercal \bm v) - \bm v^\intercal A_G \bm v \le \lambda\left(2D + \frac{(D + \sqrt{D})^2}{\abs{X}}\right) - 2D\sqrt{D}.
  \]
  As $\lambda < \sqrt{D}$, this gives a contradiction when $\abs{X} \ge M_1$ is sufficiently large.
  
  (b) Write $a = \abs{Y}$, $b = \abs{X\setminus Y}$, and $c = \abs{C_X(Y)}$. For any real numbers $\alpha,\beta,\gamma$, we consider the vector $\bm v$ that assigns $\alpha$ to the vertices in $Y$, $\beta$ to the vertices in $X\setminus Y$,  $\gamma$ to the vertices in $C_X(Y)$, and 0 to all other vertices. We have
  \[
  0\le\bm v^\intercal\left(\lambda I - A_G + \tfrac{1}{2}J\right)\bm v \le\lambda(a\alpha^2 +b\beta^2 + c\gamma^2) - 2ac\alpha\gamma + \tfrac{1}{2}(a\alpha + b\beta + c\gamma)^2
  \]
  for all real $\alpha,\beta,\gamma$. Taking $\beta = -(a\alpha + c\gamma)/(b + 2\lambda)$, we obtain
  \[
  \lambda(a\alpha^2  + c\gamma^2) - 2ac\alpha\gamma + \frac{\lambda}{b + 2\lambda}(a\alpha + c\gamma)^2 \ge 0
  \]
  for all real numbers $\alpha$ and $\gamma$. This is a quadratic form in $\alpha$ and $\gamma$. For it to take nonnegative values its discriminant must be nonpositive. Thus
  \[
  4\frac{(b+\lambda)^2}{(b+2\lambda)^2}a^2c^2 - 4\left(\lambda a + \frac{\lambda a^2}{b+2\lambda}\right)\left(\lambda c + \frac{\lambda c^2}{b+2\lambda}\right)\le 0,
  \]
  which simplifies to
  \[
  (b+\lambda)^2ac \le (\lambda a + \lambda b + 2\lambda^2)(\lambda c + \lambda b + 2\lambda^2).
  \]
  Rearranging the inequality gives
  \[
  c \le \frac{\lambda^2(a + b + 2\lambda)}{ab - \lambda^2}.
  \]
  Since $a,b$ are positive integers, we have the easy bound $ab\ge a+b-1$. Recalling that $a + b = \abs{X} \ge M_1$, we can take $M_1\ge 2\lambda^2+2$ to give the somewhat crude bound
  \[
  c \le \frac{\lambda^2(a + b + 2\lambda)}{ab - \lambda^2} \le \frac{\lambda^2(a + b + 2\lambda)}{a+b-(\lambda^2+1)} \le \frac{2\lambda^2(a + b + 2\lambda)}{a+b} = 2\lambda^2+\frac{4\lambda^3}{a+b}\le 2\lambda^2 +2\lambda.
  \]
  Choosing $M_1,M_2$ appropriately, we conclude $\abs{C_X(Y)} = c \le M_2$, as desired.
\end{proof}

\begin{proof}[Proof of \cref{thm:switch}]
  For a set of $N$ equiangular lines in $\RR^d$ with common angle $\arccos\alpha$, choose unit vectors $v_1,\dots,v_N$ in the directions of the lines arbitrarily. Let $G$ be the associated graph, whose vertex set is $V=\{v_1,\dots,v_N\}$ with an edge between two vectors if their inner product is $-\alpha$.
  
  Let $M_0 = \lceil\alpha^{-1}\rceil+2$ and define $M_1, M_2$ as in \cref{prop:independent}. By Ramsey's theorem, there exists $R = R(M_0, 2M_1)$ such that if $\abs{V} > R$, then $G$ contains either a clique of size $M_0$ or an independent set of size $2M_1$. As long as we choose $\Delta\ge R$, the result is trivially true for $\abs{V} \le R$. Thus we may assume that $\abs{V} > R$. By \cref{lem:bounded-clique}, $G$ does not contain a clique of size $M_0$. Thus $G$ must contain an independent set of size 2$M_1$, which we call $V_1$.
  
  We perform the following switching operation, modifying our set of vectors $\{v_1,\dots,v_N\}$. For any vertex $v_i\not\in V_1$ adjacent to more than half of the vertices in $V_1$, replace $v_i$ by $-v_i$.
  
  Considering how each vertex in $V \setminus V_1$ is attached to $V_1$, the set $V\setminus V_1$ can be partitioned as a disjoint union of $C_{V_1}(Y)$ as $Y$ ranges over all subsets of $V_1$ with at most $\abs{V_1}/2$ elements, since the above switching step ensures that $C_{V_1}(Y)$ is empty for $\abs{Y}>\abs{V_1}/2$. By \cref{prop:independent}(b), $\abs{C_{V_1}(Y)} \le M_2$ for each $Y\neq\varnothing$. Let $V_2=C_{V_1}(\varnothing)$, the non-neighbors of $V_1$. We know that $\abs{V\setminus V_2} \le M:= 2M_1+2^{2M_1}M_2$.
  
  It remains to bound the degree of vertices in $V$. If $v\in V_1$, then $v$ is only adjacent to vertices in $V\setminus V_2$ and thus has degree at most $M$. Now suppose $v\not\in V_1$. Let $Y$ be the set of non-neighbors of $v$ in $V_1$. The switching ensures that $\abs{Y} \ge \abs{V_1} / 2 = M_1$. Applying \cref{prop:independent}(a), the maximum degree of the subgraph induced by $C_{Y}(\varnothing)$ is at most $\lceil\lambda^2\rceil$. This set $C_{Y}(\varnothing)$ includes $V_2$ and $v$, implying that $v$ has degree at most $D := \ceil{\lambda^2} + M$. Thus we have bounded the degree of every vertex by $D$, a constant depending only on $\alpha$.
\end{proof}

\section{Further remarks}

Our main theorem completely determines $N_\alpha(d)$ for sufficiently large $d$ in the case $k(\lambda) < \infty$. However, it is still open what happens exactly when $k(\lambda) = \infty$. The construction in \cref{prop:lower-bounds} only gives a lower bound $N_\alpha(d) \ge d$, whereas the proof of \cref{thm:main} shows $N_\alpha(d) = d +O_\alpha\left(d/\log\log d\right)$. The following conjecture was made in \cite{JP20} and has been verified except when $\lambda$ is a totally real algebraic integer that is largest among its conjugates \cite[Propositions 15 and 23]{JP20}.

\begin{conjecture}\label{conj:const-error}
  Fix $\alpha \in (0,1)$, and let $\lambda = (1-\alpha)/(2\alpha)$. If $k(\lambda) = \infty$, then $N_\alpha(d) = d + O_{\alpha}(1)$.
\end{conjecture}

\begin{question}
  How large does $d$ need to be for \cref{thm:main} to hold?
\end{question}

Many interesting questions can be asked regarding \cref{thm:eigen-mult} as well.

\begin{question} \label{q:delta-max}
  Fix $\Delta$. What is the maximum possible second eigenvalue multiplicity of a connected $n$-vertex graph with maximum degree at most $\Delta$?
\end{question}

\cref{thm:eigen-mult} shows that the $\lambda_2$ multiplicity is $O_\Delta(n/\log\log n)$. On the other hand, it cannot be better than $O(n^{1/3})$ when $\Delta \ge 4$ (see the remark at the end of \cref{sec:eigen-mult}).

\begin{remark}
  It is interesting to ask the same question when restricted to Cayley graphs of finite groups. 
  For abelian or nearly abelian groups (e.g., nilpotent of bounded step), the problem of eigenvalue multiplicities has interesting connections to deep results in Riemannian geometry. 
  Following the approach of Colding and Minicozzi~\cite{CM97} on harmonic functions on manifolds and Kleiner's proof~\cite{K10} of Gromov's theorem on groups of polynomial growth \cite{Gro81}, 
  Lee and Makarychev~\cite{LM08} showed that in groups with bounded doubling constant $K = \max_{R > 0} \abs{B(2R)}/{\abs{B(R)}}$ (where $B(R)$ is the ball of radius $R$), the second eigenvalue multiplicity of such a Cayley graph is bounded, namely at most $K^{O(\log K)}$. Note that a Cayley graph on a nilpotent group of bounded step (e.g., an abelian group) with a bounded number of generators has bounded doubling constant.

  The above discussion gives a substantial improvement to \cref{thm:eigen-mult} for non-expanding Cayley graphs. 
  On the other hand, for expander graphs (not necessarily Cayley), say, satisfying $\abs{N(A)}\ge (1+c)\abs{A}$ for all vertex subsets $A$ with $\abs{A} \le n/2$, the bound in \cref{thm:eigen-mult} can be improved to $O_{\Delta, j, c}(n/\log n)$. Indeed, for such expander graphs \cref{lem:small-dominating} can be improved as follows. Every maximal $r$-separated set is necessarily an $r$-net, and the size of such a set in this expander graph is at most $n/(1+c)^{\lfloor r/2 \rfloor}$, as can be seen by considering the sizes of the $\lfloor r/2 \rfloor$-neighborhoods which must necessarily be disjoint.
  However, there are Cayley graphs that expand at some scales and have bounded doubling at other scales. Neither of these techniques applies to such graphs.
\end{remark}

The following more refined question, where we fix $\lambda > 0$, appears to be more relevant to the problem of equiangular lines, especially in pinning down the asymptotics of the error term in \cref{thm:main}.

\begin{question}
  Fix $\Delta, \lambda > 0$. What is the maximum multiplicity that $\lambda$ can appear as the second eigenvalue of a connected $n$-vertex graph with maximum degree at most $\Delta$?
\end{question}

If the answer is $O(1)$ for some $\lambda$ and sufficiently large $\Delta$, then our proof shows that \cref{conj:const-error} holds for this $\lambda$.

Finally, there are many similarly flavored questions regarding $s$-distance sets and codes in $\RR^n$, the sphere, and other spaces. Complex versions and higher dimensional analogs are also worth exploring further. We state one of these questions here, which is partially addressed in a follow-up work~\cite{JTYZZ2}.

\begin{question}
  Fix $1>\alpha\ge 0>\beta\ge -1$. What is the maximum size of a spherical $\{\alpha,\beta\}$-code in $\RR^d$? That is, what is the maximum number of unit vectors in $\RR^d$ such that all pairwise inner products are $\alpha$ or $\beta$?
\end{question}

\paragraph{\bfseries Acknowledgments.} This research was conducted while Jiang was an Instructor in Applied Mathematics at MIT, and while Yao and Zhang were participants and Tidor was a mentor in the Summer Program in Undergraduate Research (SPUR) of the MIT Mathematics Department. We thank David Jerison and Ankur Moitra for their role in advising the program. 
We thank Assaf Naor for pointing us to \cite{LM08}.

\end{document}